\documentclass{article}

\usepackage{graphicx}
\usepackage{amsfonts}
\usepackage{amsmath}
\usepackage{amssymb}
\usepackage{enumitem}
\usepackage{mathrsfs}
\usepackage{amsthm}
\usepackage{multirow}

\newtheorem*{theorem*}{Theorem}

\newtheorem{theorem}{Theorem}[section]
\newtheorem{lemma}[theorem]{Lemma}

\newtheorem{observation}[theorem]{Observation}
\newtheorem{definition}[theorem]{Definition}

\newtheorem{conjecture}[theorem]{Conjecture}

\def\Z{{\mathbb Z}}
\DeclareMathOperator{\supp}{supp}

\begin{document}

\title{Many flows in the group connectivity setting} 

\author{
  Matt DeVos\thanks{Department of Mathematics, Simon Fraser University, Burnaby, B.C. V5A 1S6.  E-mail: mdevos@sfu.ca. Supported by an NSERC Discovery Grant (Canada)}\\[1mm]
\and
  Rikke Langhede\thanks{Department of Applied Mathematics and Computer Science, Technical University of Denmark, DK-2800 Lyngby, Denmark. E-mail: rimla@dtu.dk.}
\and
  Bojan Mohar\thanks{Department of Mathematics, Simon Fraser University, Burnaby, B.C. V5A 1S6. E-mail: mohar@sfu.ca. Supported in part by the NSERC Discovery Grant R611450 (Canada) and by the Research Program P1-297 of ARRS (Slovenia).}
\and
  Robert \v{S}\'amal\thanks{Computer Science Institute (CSI) of Charles University, Malostransk\'e n\'am\v{e}st\'i 25, 118 00 Prague, Czech Republic. E-mail: samal@iuuk.mff.cuni.cz.
  Partially supported by grant 19-21082S of the Czech Science Foundation. This project has received funding from the European Union’s Horizon 2020 research and innovation programme under the Marie Sk{\l}odowska-Curie grant agreement No 823748.}
}

\maketitle

\begin{abstract}
  Two well-known results in the world of nowhere-zero flows are Jaeger's 4-flow theorem asserting that
  every 4-edge-connected graph has a nowhere-zero $\Z_2 \times \Z_2$-flow and
  Seymour's 6-flow theorem asserting that every 2-edge-connected graph has a nowhere-zero $\Z_6$-flow.
  Dvo\v{r}\'ak and the last two authors of this paper extended these results by
  proving the existence of exponentially many nowhere-zero flows under the same assumptions.
  We revisit this setting and provide extensions and simpler proofs of these results.

  The concept of a nowhere-zero flow was extended in a significant paper of Jaeger, Linial, Payan, and Tarsi to a
  choosability-type setting.  For a fixed abelian group $\Gamma$, an oriented graph $G = (V,E)$ is
  called $\Gamma$-\emph{connected} if for every function $f : E \rightarrow \Gamma$ there is a flow
  $\phi : E \rightarrow \Gamma$ with $\phi(e) \neq f(e)$ for every $e \in E$ (note that taking $f = 0$
  forces $\phi$ to be nowhere-zero).  Jaeger et al.\ proved that every oriented 3-edge-connected graph
  is $\Gamma$-connected whenever $|\Gamma| \ge 6$.  We prove that there are 
  exponentially many solutions whenever $|\Gamma| \ge 8$.  For the group $\Z_6$ we prove that
  for every oriented 3-edge-connected  $G = (V,E)$ with $\ell = |E| - |V| \ge 11$ and every
  $f: E \rightarrow \Z_6$, there are at least $2^{ \sqrt{\ell} / \log \ell}$ flows $\phi$ with
  $\phi(e) \neq f(e)$ for every $e \in E$.
\end{abstract}

\paragraph{Keywords:}{nowhere-zero flow; group connectivity; counting}
\paragraph{MSC:}{05C21; 05C30}

\section{Introduction}

Throughout this paper we permit graphs to have loops and parallel edges.  We use standard graph theory terminology and notation as in \cite{BM1976} and \cite{D}.  In particular, if $G$ is a graph and $X\subseteq V(G)$ and $S\subseteq E(G)$, we let $G-X$ and $G-S$ denote the subgraph of $G$ obtained by removing all vertices in $X$ (and their incident edges), and removing all edges in $S$ (but keeping all vertices), respectively. We also write $G[X]$ to denote the subgraph induced on $X$.

For a graph $G = (V,E)$, we define a $k$-\emph{coloring} to be a function $f : V \rightarrow \{1,2,\ldots,k\}$
with the property that $f(u) \neq f(v)$ for every $uv \in E$.  If $G$ is equipped with an orientation and $v \in V$, we let $\delta^-(v)$ denote the set of edges
that have $v$ as terminal vertex and $\delta^+(v)$ the set of edges with $v$ as initial vertex; we also put $\delta(v) = \delta^+(v) \cup \delta^-(v)$.
If $\Gamma$ is an additive abelian group, a function $\phi : E \rightarrow \Gamma$ is called a \emph{flow} or a $\Gamma$-\emph{flow} if the following rule is satisfied at every $v \in V$:
$$
 \sum_{e \in \delta^+(v)} \phi(e) = \sum_{e \in \delta^-(v)} \phi(e) .
$$
We say that $\phi$ is \emph{nowhere-zero} if $0 \not\in \phi(E)$. If $\Gamma = \Z$ and $|\phi(e)| < k$ for every $e \in E$ we call $\phi$ a $k$-\emph{flow}.
Note that if $\phi$ is a flow and we reverse the direction of an edge $e$, we may replace $\phi(e)$ with $-\phi(e)$ and this gives us a flow relative to this new orientation.
Since this operation preserves the properties of nowhere-zero and $k$-flow, the presence of a nowhere-zero $\Gamma$-flow or a nowhere-zero $k$-flow depends only
on the underlying graph and not on the particular orientation.

The study of nowhere-zero flows was initiated by Tutte \cite{bT1} who observed that these are dual to colorings in planar graphs.  Namely, he proved the following.

\begin{theorem}[Tutte \cite{bT1}]
\label{duality}
  Let $G$ and $G^*$ be dual planar graphs and orient the edges of $G$ arbitrarily.  If\/ $\Gamma$ is an abelian group with $|\Gamma| = k$, the number of $k$-colorings of $G^*$ is equal to $k$ times the number of
  nowhere-zero $\Gamma$-flows of $G$.
\end{theorem}

The above theorem has the curious corollary that for planar graphs, the number of nowhere-zero flows in an abelian group $\Gamma$ depends only on $|\Gamma|$.  Tutte proved that this holds more generally for arbitrary graphs.  That is, for any  abelian groups $\Gamma_1$ and $\Gamma_2$ with $|\Gamma_1| = |\Gamma_2|$, the number of nowhere-zero $\Gamma_1$-flows is equal to the number of nowhere-zero $\Gamma_2$-flows in every oriented graph.  Furthermore, the inherent monotonicity in coloring (every graph with a $k$-coloring has a $k'$-coloring for every $k' \ge k$) is also present in flows.  This follows from another theorem of Tutte asserting that for every $k \ge 2$, an oriented graph $G$ has a nowhere-zero $k$-flow if and only if it has a nowhere-zero $\Z_k$-flow.  In addition to establishing these fundamental properties, Tutte made three fascinating conjectures concerning nowhere-zero flows that have directed the field since then.

\begin{conjecture}[Tutte \cite{bT1,bT2,BM1976}]
  Let $G$ be an oriented 2-edge-connected graph.
  \begin{enumerate}
    \item $G$ has a nowhere-zero 5-flow.
    \item If $G$ does not have a Petersen graph minor, it has a nowhere-zero 4-flow.
    \item If $G$ is 4-edge-connected, it has a nowhere-zero 3-flow.
  \end{enumerate}
\end{conjecture}

Despite a wealth of research, all three of Tutte's conjectures remain open.  Below we have summarized some of the most significant results  to date on the presence of nowhere-zero flows.

\begin{theorem}
\label{classicflow}
  Let $G$ be an oriented graph.
  \begin{enumerate}
    \item {\rm (Seymour \cite{S})} If $G$ is 2-edge-connected, it has a nowhere-zero 6-flow.
    \item {\rm (Jaeger \cite{J})} If $G$ is 4-edge-connected, it has a nowhere-zero 4-flow.
    \item {\rm (Lov\'asz, Thomassen, Wu, Zhang \cite{LTWZ})} If $G$ is 6-edge-connected, it has a nowhere-zero 3-flow.
  \end{enumerate}
\end{theorem}

Beyond showing that a graph $G$ has a $k$-coloring, one may look to find lower bounds on the number of $k$-colorings.
Although there are infinite families of planar graphs where any two 4-colorings differ by a permutation of the colors (so there are just $4! = 24$ in total),
for 5-colorings the following theorem provides an exponential lower bound.

\begin{theorem}[Birkhoff, Lewis \cite{BL}]
  Every simple planar graph on $n$ vertices has at least $2^n$ 5-colorings.
\end{theorem}

By Theorem~\ref{duality} the above result implies that 2-edge-connected planar graphs have exponentially many nowhere-zero $\Z_5$-flows. More recently, Thomassen \cite{cT1} showed a similar exponential bound for the number of 3-colorings of triangle-free planar graphs, which yields exponentially many nowhere-zero $\Z_3$-flows in any 4-edge-connected planar graph.  In a recent work, Dvo\v{r}\'ak and the latter two authors of this paper investigated the problem of finding exponentially many nowhere-zero flows in general graphs and established the following results.

\begin{theorem}[Dvo\v{r}\'ak, Mohar, and \v{S}\'amal \cite{DMS}]
\label{flowthms}
  Let $G$ be an oriented graph with $n$ vertices and $m$ edges.
  \begin{enumerate}
    \item If $G$ is 2-edge-connected, it has $2^{2(m-n)/9}$ nowhere-zero $\Z_6$-flows.
    \item If $G$ is 4-edge-connected, it has $2^{n/250}$ nowhere-zero $\Z_2 \times \Z_2$-flows.
    \item If $G$ is 6-edge-connected, it has $2^{(n-2)/12}$ nowhere-zero $\Z_3$-flows.
  \end{enumerate}
\end{theorem}

We have revisited this topic and found shorter proofs and improved exponential bounds in the first two instances appearing in the following theorem.

\begin{theorem}
\label{ourflowthm}
  Let $G$ be an oriented graph with $n$ vertices and $m$ edges.
  \begin{enumerate}
    \item If $G$ is 3-edge-connected, it has $2^{ (m-n)/3}$ nowhere-zero $\Z_6$-flows.
    \item If $G$ is 4-edge-connected, it has $2^{n/3}$ nowhere-zero $\Z_2 \times \Z_2$-flows.
  \end{enumerate}
\end{theorem}

Theorem \ref{ourflowthm} will be proved in Section \ref{sect:3}, see Theorems \ref{flowthm1} and \ref{flowthm2}.

List-coloring provides a more general setting than that of standard coloring.  Let $G = (V,E)$ be a graph and let $\mathcal{L} = \{ L_v \}_{v \in V}$ be a family of sets indexed by the vertices of $G$.  An $\mathcal{L}$-\emph{coloring} of $G$ is a function $f$ so that $f(v) \in L_v$ holds for every $v \in V$ and $f(u) \neq f(v)$ whenever $uv \in E$.  We say that a graph $G$ is $k$-\emph{choosable} if an $\mathcal{L}$-coloring exists whenever $|L_v| \ge k$ holds for every $v \in V$.  Note that every $k$-choosable graph is necessarily $k$-colorable since we may take $L_v = \{1,2,\ldots,k\}$ for every $v \in V$.  For planar graphs, Thomassen \cite{cT3} proved the following theorem showing the existence of exponentially many ${\mathcal L}$-colorings when every set has size five.

\begin{theorem}[Thomassen \cite{cT3}]
  Let $G = (V,E)$ be a simple planar graph of order $n$.  If $\mathcal{L} = \{ L_v \}_{v \in V}$ is a family of sets with $|L_v| = 5$ for every $v \in V$,
  then there are at least $2^{n/9}$ $\mathcal{L}$-colorings of $G$.
\end{theorem}

A recent concept due to Dvo\v{r}\'ak and Postle \cite{DP} provides an even more challenging setting than list-coloring.  Let $G = (V,E)$ be a graph with an arbitrary orientation, fix a positive integer $k$ and
for every edge $e \in E$ let $\sigma_e$ be a permutation of $\{1,2,\ldots,k\}$.  We define a  \emph{DP-coloring} of this system to be a function $f : V \rightarrow \{1,2,\ldots,k\}$ with the property that for
every directed edge $e = (u,v)$ we have $f(u) \neq \sigma_e( f(v) )$.  The \emph{DP-coloring number} of $G$ is the smallest $k$ so that such a coloring exists for every assignment of permutations.  Every graph
with DP-coloring number $k$ is also $k$-choosable. To see this, let $\{ L_v \}_{v \in V}$ be a family of sets with $|L_v| = k$ for every $v \in V$ and for every oriented edge $e = (u,v)$ of $G$ choose a
bijection $\sigma_e: L_u\to L_v$ such that $\sigma_e$ is the identity when restricted to $L_u \cap L_v$.  A DP-coloring for the family $\{ \sigma_e\}_{e \in E}$ then gives a list-coloring.

There is a natural group-valued specialization of DP-coloring.  Again, let $G = (V,E)$ be an oriented graph and let $\Gamma$ be an additive abelian group.
For every $e \in E$, let $\gamma_e \in \Gamma$.  Then a function $c : V \rightarrow \Gamma$ may be considered a coloring if for every $e = (u,v) \in E$ we have $c(u) \neq \gamma_e + c(v)$.
So here we are operating in the realm of DP-coloring, but using a group as the ground set, and taking permutations associated with group elements on the edges.
This group-valued DP-coloring is usually called group coloring. Its dual was investigated many years ago by Jaeger, Linial, Payan, and Tarsi~\cite{JLPT}, who gave the following definition.

\begin{definition}
  Let $G = (V,E)$ be an oriented graph and let $\Gamma$ be an abelian group.  We say that $G$ is \emph{$\Gamma$-connected} if it satisfies the following property:
  For every $f : E \rightarrow \Gamma$ there is a flow $\phi : E \rightarrow \Gamma$ satisfying $\phi(e) \neq f(e)$ for every $e \in E$.
\end{definition}

In fact, the authors showed in \cite{JLPT} that the above concept has several equivalent formulations.  Our interest here will be in counting the number of functions $\phi$ satisfying the above property, and for this purpose
all of the equivalent forms of group-connectivity operate the same.  So for simplicity we shall stay with the above definition.  Using the basic method of Jaeger et al.\ we will prove the following theorem in Section~\ref{sect:5}.

\begin{theorem}
\label{order8plus}
Let $G=(V,E)$ be an oriented 3-edge-connected graph with $\ell = |E| - |V|$ and let $\Gamma$ be an abelian group with $|\Gamma| = k \ge 6$.  For every $f : E \rightarrow \Gamma$ we have
\[
  | \{ \phi : E \rightarrow \Gamma \mid \mbox{$\phi$ is a flow and $\phi(e) \neq f(e)$ for every $e \in E$} \}|
	\ge \left\{ \begin{array}{cl}
          \frac{1}{2} (\frac{k-6}{2})^{\ell}	&	\mbox{if $k$ is odd,}	\\
          \frac{1}{2} (\frac{k-4}{2})^{\ell}	&	\mbox{if $k$ is even.}
		\end{array} \right.
\]
\end{theorem}

For $k \ge 8$ the above theorem gives us an exponential number of flows, but for $k=6,7$ it only implies the existence of one.  We believe that this represents a shortcoming of our techniques and suspect the
result holds as well when $k=6$,~$7$.

\begin{conjecture}\label{conjexp}
  There exists a fixed constant $c > 1$ so that the following holds.  For every 3-edge-connected oriented graph $G = (V,E)$ of order $n$,
  every abelian group $\Gamma$ with $|\Gamma| \ge 6$, and every $f : E \rightarrow \Gamma$, there exist at least $c^n$ flows $\phi : E \rightarrow \Gamma$ with $\phi(e) \neq f(e)$ for every $e \in E$.
\end{conjecture}

Theorem~\ref{order8plus} shows the above conjecture is true for all abelian groups except for $\Z_6$ and $\Z_7$.  In the case of $\Z_7$ we have no interesting lower bound.  In the case of $\Z_6$, we have the following result showing a super-polynomial lower bound, which has the flavor of a similar coloring result by Asadi et al.~\cite{ADPT}. Throughout we use $\log$ to denote the logarithm base~$2$.

\begin{theorem}
\label{order6}
  Let $G=(V,E)$ be an oriented 3-edge-connected graph with $\ell = |E| - |V| \ge 11$ and let $f : E \rightarrow \Z_6$.
  There exist at least $2^{\sqrt{\ell} / \log \ell}$ flows $\phi: E \rightarrow \Z_6$ with the property that $\phi(e) \neq f(e)$ for every $e \in E$.
\end{theorem}

The rest of this paper is organized as follows.  In the  next section we show that many of our counting problems can be reduced to cubic graphs. In the third section we establish Theorem~\ref{ourflowthm} on the
existence of many nowhere-zero flows.  Section~\ref{sec:4} introduces Seymour's concept of a $k$-base and then this is used in Section~5 to prove Theorem~\ref{order8plus}.  Sections~\ref{sec:6} and~\ref{sec:7} develop the techniques required
to find many decompositions of a 3-connected cubic graph into a 1-base and a 2-base, and then in Section\ref{sec:8} this is exploited to prove Theorem~\ref{order6}. The results about 1-base / 2-base decompositions (Lemmas \ref{lem:7.2} and \ref{manydecomps}) may be of independent interest.

\section{Reduction to cubic graphs}
\label{sec:2}

As is common in the world of nowhere-zero flows, certain basic operations will reduce some of our problems to the setting of cubic graphs.  If $G = (V,E)$ is a graph and $X, Y \subseteq V$ are disjoint we let $E(X,Y) = \{ xy \in E \mid \mbox{$x \in X$ and $y \in Y$} \}$.  For $z \in V$ with $z \not\in Y$ we let $E(z,Y) = E( \{z\}, Y)$.

\begin{lemma}
\label{reduce2cubic}
  Let $G = (V,E)$ be a 3-edge-connected graph with $|E| - |V| = \ell$. Then there exists a 3-edge-connected cubic graph $G'$ with $|V(G')| = 2 \ell$ so that $G$ can be obtained from $G'$ by contracting the edges of a forest.
\end{lemma}

\begin{proof}
  We proceed by induction on $S = \sum_{v \in V} \left( \mathrm{deg}(v) - 3 \right)$.  In the base case this parameter is 0 and the result holds
  trivially by setting $G'  = G$.  For the inductive step we may choose $v \in V$ with $\mathrm{deg}(v) > 3$.  For $e,e' \in \delta(v)$
  with $e \neq e'$ we may form a new graph $G'$ from $G$ by adding a new vertex $v'$, changing $e$ and $e'$ to have $v'$ as one end instead
  of $v$ (in case of a loop at $v$, change just one end) and then adding a new edge $f$ with ends $v,v'$.  We say that  $G'$ is obtained
  from $G$ by \emph{expanding relative to} $e,e'$.  Note that we may return from $G'$ to $G$ by contracting the newly added edge $f$.
  If we can construct a new graph $G'$ by expanding so that $G'$ is still $3$-edge-connected, then the result follows by applying
  induction to $G'$, so it suffices to prove this. (Note that we have added one vertex and one edge, so $\ell$ stays the same, while the sum~$S$ decreases.)

  If $G - v$ is disconnected, there must be at least $3$ edges between $v$ and each component of $G - v$.  Now choosing $e,e' \in \delta(v)$
  so that $e$ and $e'$ have ends in different components of $G-v$ and expanding relative to $e,e'$ preserves $3$-edge-connectivity.  So we
  may assume $G-v$ is connected.  If $G-v$ is 2-edge-connected, then any expansion suffices, so we may assume otherwise and choose a
  minimal nonempty set $X \subseteq V(G-v)$ so that there is just one edge between $X$ and $Y = V(G-v) \setminus X$.  Since $G$ is
  3-edge-connected, there must be at least two edges between $v$ and $X$ and at least two between $v$ and $Y$.  Choose $e \in E(v,X)$ and
  $e' \in E(v,Y)$ and then expanding relative to $e,e'$ gives us a 3-edge-connected graph $G'$ and this completes the proof.
\end{proof}

\begin{observation}
\label{forestdiff}
  Let $G = (V,E)$ be an oriented graph, let $F \subseteq E$ and assume that $(V,F)$ is an (oriented) forest.  Let $\Gamma$ be an abelian group, and let $\phi_1, \phi_2 : E(G) \rightarrow \Gamma$ be flows.  If $\phi_1(e) = \phi_2(e)$ holds for every $e \in E \setminus F$, then $\phi_1 = \phi_2$.
\end{observation}

\begin{proof}
  The function $\phi_1 - \phi_2$ is also a flow and must be identically 0 since the support of every flow is a union of (edge-sets of) cycles.
\end{proof}

\section{Many nowhere-zero flows}
\label{sect:3}

In this section we prove Theorem~\ref{ourflowthm} concerning the existence of many nowhere-zero flows.

\begin{lemma}\label{supportflow}
  Let $G = (V,E)$ be a graph with $|V| = n$ and $|E|=m$.  Let $\phi : E \rightarrow \Z_2$ and $\psi : E \rightarrow \Z_k$ be flows with $\supp(\phi) \cup \supp(\psi) = E$ and let $t = | \supp(\phi)|$.
  Then $G$ has at least $2^{m-n-t/k}$ nowhere-zero $\Z_2 \times \Z_k$-flows.
\end{lemma}

\begin{proof}
The support of $\phi$ may be expressed as a disjoint union $\sqcup_{i=1}^k C_i$ where each $C_i$ is the edge-set of a cycle.
Let us fix an arbitrary orientation of $G$. (In the sequel we will no longer be mentioning that we have an implicit chosen orientation of the graph whenever we speak of flows in the graph.)
For every $1 \le i \le k$ there is a $\Z_k$-flow $\rho_i$ of $G$ with support $C_i$ that has values $\pm 1$ on every edge in $C_i$.
By adding a suitable multiple of $\rho_i$ to $\psi$ we may assume that at most $\frac{1}{k} |E(C_i)|$ edges of $C_i$ are not in the support of $\psi$.
After applying this operation to each $C_i$, the resulting $\Z_k$-flow $\psi'$ will still have $\supp(\phi) \cup \supp(\psi') = E$ but will additionally satisfy $|\supp(\psi')| \ge m - t/k$.

Let $E' = \supp(\psi')$, let $G' = (V,E')$ and note that the dimension of the cycle space of $G'$ is at least $|E'| - |V| \ge m - n - t/k$.   It follows that the number of $\Z_2$-flows of $G$ supported on
a subset of $E'$ is at least $2^{m-n - t/k}$.  If we take any such $\Z_2$-flow, say $\eta$, the mapping from $E$ to $\Z_2 \times \Z_k$ given by $e \mapsto (\phi(e) + \eta(e), \psi'(e) )$ will be a nowhere-zero
$\Z_2 \times \Z_k$-flow of $G$ and this gives the desired count.
\end{proof}

With this lemma in hand we are ready to prove Theorem~\ref{ourflowthm}.  The first part of this is given by Theorem~\ref{flowthm1} and the second by Theorem~\ref{flowthm2}.

\begin{theorem}
\label{flowthm1}
Every oriented 2-edge-connected graph with $n$ vertices and $m$ edges has at least $2^{(m-n)/3}$ nowhere-zero $\Z_2 \times \Z_3$-flows.
\end{theorem}

\begin{proof}
  We proceed by induction on $n$ with the base case $n=1$ holding trivially (each of the $m$ loop edges may be assigned any non-zero value in $\Z_2 \times \Z_3$ to get a nowhere-zero flow and this can be done in
  $5^m \ge 2^{(m-1)/3}  = 2^{(m-n)/3}$ ways).  Next suppose that there is an edge $e$ so that $e$ is in a 2-edge-cut with another edge $e'$, and apply the theorem inductively to $G/e$.  Every nowhere-zero flow $\phi$ of
  $G/e$ extends uniquely to one in $G$ (if $e$ and $e'$ are consistently oriented in the edge-cut $\{e,e'\}$, then we extend by setting $\phi(e) = - \phi(e')$).  This gives the desired count of nowhere-zero flows,
  and we may therefore assume that $G$ is 3-edge-connected.

  Now apply Lemma~\ref{reduce2cubic} to choose a 3-edge-connected cubic graph $G' = (V',E')$ with $2 \ell$ vertices where $\ell = m-n$ so that $G$ can be obtained from $G'$ by contracting the edges of a forest.  In
  light of Observation~\ref{forestdiff}, it suffices to prove that $G'$ has at least $2^{(m-n)/3} = 2^{\ell /3}$ nowhere-zero $\Z_2 \times \Z_3$-flows.  By Seymour's 6-flow theorem (part 1 of
  Theorem~\ref{classicflow}) we may choose flows $\phi : E' \rightarrow \Z_2$ and $\psi : E' \rightarrow \Z_3$ with $\supp(\phi) \cup \supp(\psi) = E'$.  Since $G'$ is cubic, the support of $\phi$ is a
  disjoint union of (edge-sets of) cycles and thus $|\supp(\phi)| \le |V'| = 2 \ell$.  Now applying Lemma~\ref{supportflow} gives us at least $2^{\ell - 2\ell/3} = 2^{\ell/3}$ nowhere-zero $\Z_2 \times \Z_3$-flows,
  as claimed.
\end{proof}

\begin{theorem}
\label{flowthm2}
  Every oriented 4-edge-connected graph with $n$ vertices has at least $2^{n/3}$ nowhere-zero $\Z_2 \times \Z_2$-flows.
\end{theorem}

\begin{proof}
Let $G = (V,E)$ be a 4-edge-connected graph with $|V| = n$ and $|E| = m$ and apply Jaeger's 4-flow theorem to choose a nowhere-zero flow $\phi : E \rightarrow \Z_2 \times \Z_2$.  Thanks to the symmetry of the factors of the
  group $\Z_2 \times \Z_2$ we may permute these flow values $(1,0), (0,1), (1,1)$ while maintaining a flow (for instance, we can change all edges with flow value $(0,1)$ to $(1,1)$
  and all those with flow value $(1,1)$ to $(0,1)$ and we still have a flow).  By way of this operation, we may assume that at least $\frac{1}{3}m$ edges $e$ satisfy $\phi(e) = (0,1)$.  Now letting
  $\phi_1 : E \rightarrow \Z_2$ denote the first coordinate of $\phi$ and $\phi_2 : E \rightarrow \Z_2$ the second, we have $t = | {\mathit supp}(\phi_1)| \le \frac{2}{3}m$.
  Since every vertex has degree at least 4 we have $m \ge 2n$ and thus
$$
  m - n - \tfrac{1}{2}t \ge \tfrac{2}{3}m - n \ge \tfrac{1}{3}n.
$$
The result follows from this inequality and Lemma~\ref{supportflow} applied to $\phi_1$ and $\phi_2$.
\end{proof}

\section{Seymour's bases}
\label{sec:4}

Following Seymour~\cite{S}, we introduce a closure operator for subsets of edges of a graph.  Let $G = (V,E)$ be a graph and let $k$ be a positive integer.  For every $S \subseteq E$, the \emph{$k$-closure} of $S$, denoted $\langle S \rangle_k$, is defined to be the minimal edge-set $R$ such that $S\subseteq R \subseteq E$ and every cycle $C \subseteq G$ with $E(C) \setminus R \ne \emptyset$, has at least $k+1$ edges that are not in $R$.
Clearly, $\langle S \rangle_k$ can be constructed starting with $R=S$ and then consecutively add the edges of any cycle $C$ violating the condition, i.e., if $0 < |E(C)\setminus R| \le k$, then we add $E(C)$ into $R$ and repeat.
This is a closure operator as evidenced by the following properties, all of which are easy to verify:
\[
  S \subseteq \langle S \rangle_k, \qquad \langle \langle S \rangle_k \rangle_k = \langle S \rangle_k, \qquad S' \subseteq S \Rightarrow \langle S' \rangle_k \subseteq \langle S \rangle_k.
\]

A subset $S \subseteq E$ is called a $k$-base if $\langle S \rangle_k = E$. It is easy to see that a set of edges of a connected graph is a 1-base if and only if it contains the edge-set of a spanning tree.  A key
feature of the definition of $k$-bases is that they can be extended to the whole edge-set by a sequence of steps, each of which adds at most $k$ new edges along a cycle.
This property can be used to find $\Gamma$-flows whose support contains $E\setminus S$ for all groups~$\Gamma$ with $|\Gamma|>k$. In particular, we have the following
well-known statement:

\begin{lemma}
\label{base2groupflow}
  Let $k$ and $q$ be integers with $0 < k < q$, let $G = (V,E)$ be a graph, and let $S \subseteq E$ be a $k$-base. For every $f : E \setminus S \rightarrow \Z_q$ there exists a flow $\phi : E \rightarrow \Z_q$ satisfying $\phi(e) \neq f(e)$ for every $e \in E \setminus S$.
\end{lemma}

\begin{proof}
  Since every $k$-base is a $(q-1)$-base, we may assume that $q=k+1$. 
  We proceed by induction on $|E \setminus S|$. As a base case, when $|E \setminus S| = 0$ the result holds trivially (with $\phi$ zero everywhere).
  For the inductive step $S \neq E$ so we may choose a cycle $C \subseteq G$ so that $0 < | E(C) \setminus S| \le k$ and we let $S' = S \cup E(C)$.
  By applying the induction hypothesis to $S'$ we may choose a flow $\phi' : E \rightarrow \Z_{k+1}$ so that $\phi'(e) \neq f(e)$ holds for every
  $e \in E \setminus S'$.  Now let $\nu : E \rightarrow \Z_{k+1}$ be a $\Z_{k+1}$-flow taking the values $\pm 1$ on $E(C)$ and the value $0$ elsewhere.
  Let $x \in \Z_{k+1}$ and consider the flow $\phi = \phi' + x \nu$. For every $e \in E(C) \setminus S$ there is precisely one value of $x$ for which we will have $\phi(e) = f(e)$.
  Since there are at most $k$ edges in $E(C) \setminus S$ but $k+1$ possible values for $x$, we may assign $x$ a value so that $\phi(e) \neq f(e)$ holds
  for every $e \in E(C) \setminus S$, and then $\phi$ satisfies the lemma.
\end{proof}

In Seymour's original paper~\cite{S} on nowhere-zero 6-flows, he proves two structure theorems giving the existence of a 2-base (these give two different proofs of the 6-flow theorem).  We summarize these results below.

\begin{theorem}[Seymour \cite{S}]
\label{seymourdecomp}
Let $G$ be a 3-edge-connected cubic graph.
\begin{enumerate}
  \item There exists a collection of edge-disjoint cycles $C_1, \ldots, C_t$ so that $\cup_{i=1}^t E(C_i)$ is a 2-base.
  \item There exists a partition of $E(G)$ into $\{B_1, B_2\}$ so that $B_i$ is an $i$-base for $i=1,2$.
\end{enumerate}
\end{theorem}

Both parts of this theorem give the existence of a nowhere-zero $\Z_2 \times \Z_3$-flow quite immediately. For the first, we may choose a $\Z_2$-flow with support equal to $\cup_{i=1}^t E(C_i)$ and Lemma \ref{base2groupflow}
(applied with $f = 0$) gives a $\Z_3$-flow whose support contains the complement.  For the second part, we may apply the above lemma twice (both times with $f=0$) to choose a $\Z_2$-flow with support containing
$E \setminus B_1$ and a $\Z_3$-flow with support containing $E \setminus B_2$.

The second decomposition in Theorem~\ref{seymourdecomp} will be more useful to us, and in fact we will require some slightly stronger variants of this, so we will develop a proof of this later in the paper.
To see the utility of this second decomposition in the setting of group-connectivity, we follow Jaeger et al.~\cite{JLPT} to prove the following result.

\begin{theorem}[Jaeger, Linial, Payan, and Tarsi \cite{JLPT}]
  Every 3-edge-connected graph is $\Z_6$-connected.
\end{theorem}

\begin{proof}
  Let $G = (V,E)$ be an oriented 3-edge-connected graph, let $f_1 : E \rightarrow \Z_2$, and let $f_2 : E \rightarrow \Z_3$.  Apply the second part of Theorem~\ref{seymourdecomp} to choose a partition of $E$ into $\{B_1, B_2\}$ so that $B_i$ is an $i$-base for $i=1,2$.  Now for $i=1,2$ we apply Lemma~\ref{base2groupflow} to choose a flow $\phi_{i} : E \rightarrow \Z_{i}$ so that $\phi_i(e) \neq f_i(e)$ holds for every $e \in E \setminus B_i$.  Now $(f_1(e), f_2(e)) \neq (\phi_1(e), \phi_2(e))$ holds for every $e \in E$ and this completes the proof.
\end{proof}

\section{Large groups}
\label{sect:5}

In this section we prove Theorem~\ref{order8plus}, our result for groups of order 8 or more.  For such large groups all that we need is Seymour's decomposition theorem.
We have restated this theorem below for convenience.

\begin{theorem*}
Let $G=(V,E)$ be an oriented 3-edge-connected graph with $\ell = |E| - |V|$ and let $\Gamma$ be an abelian group with $|\Gamma| = k \ge 6$.  For every $f : E \rightarrow \Gamma$ we have
\[ | \{ \phi : E \rightarrow \Gamma \mid \mbox{$\phi$ is a flow and $\phi(e) \neq f(e)$ for every $e \in E$} \}|
	\ge \left\{ \begin{array}{cl}
		\frac{1}{2} (\frac{k-6}{2} )^{\ell}	&	\mbox{if $k$ is odd,}	\\
		\frac{1}{2}( \frac{k-4}{2} )^{\ell}	&	\mbox{if $k$ is even.}
		\end{array} \right. \]
\end{theorem*}

\begin{proof}
  Apply Lemma~\ref{reduce2cubic} to choose a cubic graph $G' = (V', E')$ with $|V'|= 2 \ell$ and note that by Observation~\ref{forestdiff} it suffices to prove the above result with $G'$ in place of $G$.
  If $k$ is even, choose $x \in \Gamma$ to be an element of order 2; otherwise choose $x \in \Gamma \setminus \{0\}$.

  Apply the second part of Theorem~\ref{seymourdecomp} to choose a 2-base $B$ of $G'$ such that $E'\setminus B$ is a 1-base. This implies that $G' - B$ is connected.  The first stage in our proof will be to construct many flows $\phi : E' \rightarrow \Gamma$ with
  the following property for every $e \in E' \setminus B$:
\[
  (\star) \quad \phi(e) \not\in \{ f(e), f(e) + x, f(e) -x \}.
\]
  Let us observe that $f(e)+x=f(e)-x$ when $k$ is even.
  Let us now invoke the definition of 2-base to choose a sequence of nested sets
  $B = B_0 \subset B_1 \subset \ldots \subset B_t = E'$ satisfying:
  \begin{itemize}
    \item $|B_{i} \setminus B_{i-1}| \le 2$ for every $1 \le i \le t$, and
    \item there exists a cycle $C_i$ with $(B_i \setminus B_{i-1}) \subseteq E(C) \subseteq B_i$ for every $1 \le i \le t$.
  \end{itemize}
  We will construct our flows recursively using elementary flows on the cycles $C_i$ (working backwards).  Initially start with $\phi : E' \rightarrow \Gamma$ to be the zero flow.  Let $\nu_t$ be either the zero flow or an elementary flow supported on $E(C_t)$ (so there are $k$ choices for $\nu_t$) and modify $\phi$ by adding $\nu_t$ to it.  If $k$ is odd (even), there are at most 3 (2) possible choice of $\nu_t$ so that $(\star)$ fails on an edge in $B_t\setminus B_{t-1}$. Thus, there at least $k-6$ $(k-4)$ ways to choose $\nu_t$ so that condition $(\star)$ is satisfied on every edge in $B_t \setminus B_{t-1}$.  Next choose $\nu_{t-1}$ to be either the zero flow or an elementary flow supported on $E(C_{t-1})$ and modify $\phi$ by adding $\nu_{t-1}$ so that $(\star)$ is satisfied on every edge in $B_{t-1} \setminus B_{t-2}$, and continue in this manner.  Since the edges in $B_{i} \setminus B_{i-1}$ satisfy $(\star)$ at the point when we add the flow $\nu_i$ and these edges do not appear in the support of $\nu_{i-1}, \nu_{i-2},  \ldots, \nu_1$, at the end of this process we have a flow $\phi$ that satisfies $(\star)$ on every edge in $E' \setminus B$.  Since $G' - B$ is connected we have $|E' \setminus B| \ge 2 \ell - 1$ and this means that $t \ge \ell$.  Therefore, the number of flows $\phi$ satisfying $(\star)$ on every edge in $E' \setminus B$ is at least $(k-6)^{\ell}$ when $k$ is odd and at least $(k-4)^{\ell}$ when $k$ is even.

  Choose a spanning tree $T$ with $E(T) \subseteq E' \setminus B$.  For every edge $e \in B$ let $C_e$ be the edge-set of the fundamental cycle of $e$ with respect to $T$.  For every $S \subseteq B$ define $\hat{S} = \bigoplus_{e \in S} C_e$ where $\bigoplus$ denotes the symmetric difference.  The set $\hat{S}$ may be expressed as a disjoint union of (edge-sets of) cycles so we may choose a flow $\mu_S : E' \rightarrow \Gamma$ supported on $\hat{S}$ so that $\mu_S(e) = \pm x$ for every $e \in S$.  Now for every flow $\phi$ satisfying $(\star)$ on every edge in $E' \setminus B$ we let $S = \{ e \in B \mid \phi(e) = f(e) \}$ and we define $\phi' = \phi + \mu_S$.  It follows from this construction that the resulting flow $\phi'$ will satisfy $\phi'(e) \neq f(e)$ for every $e \in E'$.  Since $|B| \le \ell + 1$, the number of subsets $S \subseteq B$ is at most
  $2^{\ell + 1}$, so if $k$ is odd we have at least $\frac{1}{2} \left( \frac{k-6}{2} \right)^{\ell}$ flows $\phi$ with the same set $S$ and for $k$ even this count will be $\frac{1}{2} \left( \frac{k-4}{2} \right)^{\ell}$.  Since each of these flows is modified by adding the same flow, $\mu_S$, this gives us the desired number of flows $\phi'$.
\end{proof}

\section{Peripheral paths and cycles}
\label{sec:6}

  Seymour's proof of the second part of Theorem~\ref{seymourdecomp} is based on an iterative procedure during which the edge partition is formed, and we will require a strong form of this.  A key concept in this
  process is that of a path or cycle who's removal leaves the graph connected.  The purpose of this section is to prove three lemmas that provide the tools we need to find such paths and cycles.

  Tutte~\cite{Tu63} called a cycle $C$ of a graph $G$ \emph{peripheral} if $C$ is induced and $G - V(C)$ is connected.  A key feature of peripheral cycles is that for a graph $G$ embedded in the plane, every peripheral cycle must
  bound a face.  Tutte proved that for a 3-connected graph, every edge is contained in at least two peripheral cycles, thus giving an abstract characterization of the faces of a 3-connected planar graph (they are
  precisely the peripheral cycles).  For our purpose, we will be interested only in subcubic graphs and in this setting we can use the following definition of peripheral edge-sets.

\begin{definition}
  If $G = (V,E)$ is a subcubic graph, an edge-set $S \subseteq E$ is \emph{peripheral} if $G - S$ is connected.  We call a subgraph $H \subseteq G$ \emph{peripheral} if $E(H)$ is peripheral.
\end{definition}

Note that with this definition, Seymour's second decomposition theorem asserts the existence of a peripheral 2-base in every 3-edge-connected cubic graph.
Below we state a restricted form of Tutte's theorem (for cubic graphs) of use to us. We will provide a proof of this below.

\begin{theorem}[Tutte \cite{Tu63}]
\label{tutteperiphcubic}
  Let $G$ be a 3-edge-connected cubic graph.  For any two edges of $G$ incident with the same vertex, there exists a peripheral cycle containing both of them.
\end{theorem}

Now we are ready for the first of the lemmas from this section.

\begin{lemma}
\label{getapath1}
  Let $G = (V,E)$ be a 3-edge-connected cubic graph, let $X \subset V$ be nonempty, let $H$ be a component of $G - X$, and let $f \in E(X, V(H))$.  Then there exists a (possibly trivial) path $P \subseteq H$ with
  ends $y_0, y_1$ satisfying:
\begin{itemize}
  \item $P$ is peripheral in $H$;
  \item there exist distinct edges $e_0, e_1$, where $e_i \in E(y_i, X)\setminus \{f\}$ for $i=0,1$;
  \item no internal vertex on the path $P$ has a neighbor in~$X$. 
\end{itemize}
Moreover, if $H$ is 2-edge-connected and we prescribe any $y_0 \in V(H)$ such that $E(y_0,X) \setminus \{f\}$ is nonempty, we can still guarantee a peripheral path as above.
\end{lemma}

\begin{proof}
  If $H$ is a single vertex, then the result is obviously true.
  If $H$ is not 2-edge-connected, then consider its block structure. 
  If one of the leaf blocks of $H$ is just a cut-edge $vy$, where $y$ has degree 1 in $H$ and is not incident with $f$, then we take $y_0=y_1=y$ and $P=y$, and the result follows. 
  Otherwise, we let $H'$ be a leaf block that is not incident with~$f$. Further, we let $f'$ be the only cut-edge of~$H$ incident with~$H'$.
  We extend $X$ to $X' = X \cup (V(H) \setminus V(H'))$. If we find a path~$P$ for $X'$, $H'$ and~$f'$, then the same path works for $X$, $H$ and $f$ 
  because $f'$ is the only edge in $E(H',X') \setminus E(H,X)$ and the path~$P$ does not contain the vertex $z'$ (the end of~$f'$ in~$H'$). 
  We define $Y = \{ y \in V(H') \mid E(y,X') \setminus \{f'\} \neq \emptyset \}$ and pick arbitrary $y_0 \in Y$.

  If, on the other hand, $H$ was 2-edge-connected (and $y_0$ was specified), then we put $H' = H$, $X' = X$ and $f' = f$; we write $f' = x'z'$ with $x' \in X'$.
  In both cases we now have $H'$ 2-edge-connected and we want to find a peripheral path~$P$ in it with one end specified. Note that since $H'$ is 2-connected, the end $z'$ of $f'$ in $H'$ is not in $Y$, and by the third property it should not be on the path.

  As we have dealt with the case of $H'$ being a single vertex, we now have the useful property that each vertex of~$H'$ has at most one incident edge going to~$X'$.
  Choose a path $P \subseteq H' - z'$ starting at~$y_0$ with the other end in $Y \setminus \{y_0\}$,
  subject to the following conditions:

\begin{enumerate}[label=(\roman*)]
  \item The component of $H' - E(P)$ containing $z'$ has maximum size.
  \item Subject to (i), the lexicographic ordering of the sizes of the components of $H' - E(P)$ not containing $z'$ is maximum (i.e., the largest component not containing $z'$ has maximum size, and subject to this
    the second largest has maximum size, and so on).
\end{enumerate}

  The assumption that $G$ is 3-edge-connected implies that $|Y| \ge 2$, and as $H'$ is 2-edge-connected, there exists some path~$P$. Let $y_1$ be the other end of $P$.
  Note that our path $P$ chosen according to the above criteria has no interior vertices in $Y$, otherwise we may take a subpath.
  We claim that $P$ is peripheral in $H'$.  Suppose (for a contradiction) that this is not the case and
  let $F$ be a component of $H' - E(P)$ such that $z' \notin V(F)$ and $F$~is of minimum possible size.
  Define $P'$ to be the minimal subpath of~$P$ that contains all vertices of $V(F) \cap V(P)$.

  Suppose there exists a vertex $q \in V(P')$ that is not contained in $V(F)$.
  In this case we may choose a path $P^* \subseteq F$ with the same ends as $P'$.
  As $P^*$ avoids~$q$, modifying our original path $P$ by replacing the subpath $P'$ with $P^*$ gives us a path that contradicts the choice of~$P$: the component containing~$q$ increases, 
  all others except~$F$ stay the same or increase. 

  Thus we must have $V(P') \subseteq V(F)$. 
  If $P'$~shares an end with~$P$ then $F$ is connected to the rest of~$H'$ by a single edge, a contradiction with 2-edge-connectivity of~$H'$.

  Suppose next, there exists a vertex $w \in Y \cap V(F)$. By now we know, that $w$ is not an end of $P$. We choose a path $Q \subseteq F$ with one end $w$ and the other one the first vertex of $P'$; 
  note that $z' \notin V(Q)$ by the choice of $P$ and~$F$.
  Now $P \cup Q$ contains a path $\tilde P$ that contradicts the choice of $P$ relative to (i) or~(ii): components of~$H'-E(\tilde P)$ are larger than or equal to
  the corresponding components of~$H' - E(P)$, except for~$F$, which was the least significant in our selection process
  and the component containing~$y_1$ becomes strictly larger.  
  Therefore, no such vertex~$w$ can exist.  

  It follows that $F$~is connected to the rest of~$G$ by only two edges, a contradiction with 3-edge-connectivity of~$G$.
  We deduce that $P$ is peripheral in $H'$, and this completes the proof.
\end{proof}

Before we prove our strong form of the above result let us pause to prove Tutte's peripheral cycles theorem for cubic graphs using Lemma \ref{getapath1}.

\begin{proof}[Proof of Theorem~\ref{tutteperiphcubic}]
  Let $x$ be the vertex incident with both prescribed edges $e_0,e_1$.
  Let $X = \{x\}$, let $f \in \delta(x)\setminus \{e_0,e_1\}$ and apply Lemma~\ref{getapath1} for the set $X$ and the edge $f$.
  This gives us a peripheral path $P$ of $G - x$. The cycle $C$ formed by adding the vertex~$x$ and
  the edges $\{e_0,e_1\}$ is a peripheral cycle in $G$.
\end{proof}

Next we establish a stronger version of the above lemma that will provide us with some choice in our basic process.
This is a key ingredient for us in proving the existence of many flows in the group $\Z_6$.
The proof has similar basic structure as the proof of Lemma~\ref{getapath1}; with a few more subtleties -- including using Lemma~\ref{getapath1} in one of the steps.
We recall that if $P$ is a path containing vertices $a$, $b$, then $aPb$ denotes the subpath of~$P$ from~$a$ to~$b$.

\begin{lemma}
\label{getapath2}
Let $G = (V,E)$ be a cyclically 4-edge-connected cubic graph, let $X \subset V$ have $|X| \ge 2$, let $H$ be a component of $G-X$, and let $f =xz \in E$ have $x \in X$ and $z \in V(H)$.
If $| E(v,X) \setminus \{f\}| \le 1$ for every $v \in V(H)$, then there exist distinct vertices $y, y_1, y_2 \in V(H)\setminus \{z\}$ such that $E(y,X) \neq \emptyset$ and $E(y_i, X) \neq \emptyset$ for $i=1,2$, and for $i=1,2$ there exists a path $P_i \subseteq H-z$ with ends $y$ and $y_i$ that is peripheral in $H$ and contains no internal vertices with a neighbor in $X$. Moreover, the edge of $P_1$ incident with $y$ is distinct from the edge of $P_2$ incident with $y$.
\end{lemma}

\begin{proof}
  As in the proof of Lemma~\ref{getapath1}, we first suppose that the graph $H$ has a cut-edge, and therefore a nontrivial block structure.
  In this case, choose $H'$ to be a leaf block of~$H$ that does not contain $z$.
  The condition that $| E(v,X) \setminus \{f\}| \le 1$ for every $v \in V(H)$ implies that $H'$ is nonempty, is not just a vertex, and has no vertices of degree 1 except possibly $z$. This implies that $H'$ is 2-connected. 
  Let $z' \in V(H')$ be the unique vertex of~$H'$ incident with a cut-edge of $H$ and let $X' = X \cup (V(H) \setminus V(H'))$.
  In the case that our graph $H$ has no cut-edge, then we set $H' = H$, set $z' = z$, and set $X' = X$.
  Observe that to complete the proof of the lemma, it suffices to solve the problem with $H'$, $z'$, $X'$ in place of $H$, $z$, and $X$.
  This adjustment has granted us the useful property that $H'$ is 2-connected.

  Set $Y = \{ y \in V(H') \mid E(y,X') \neq \emptyset \}$ and note that $|Y| \ge 4$ as neither $X'$ nor~$H'$ can be a single vertex and $G$ is cyclically 4-edge-connected. Also, let $Y' = Y\setminus \{z'\}$. 
  Declare a nontrivial path $P \subseteq H' - z'$ to be \emph{good} if $P$ is peripheral in~$H'$, both ends of~$P$ are in~$Y$, and no interior vertex of~$P$ is in~$Y$. 
  Let $S \subseteq E(H')$ be the set of edges incident with a vertex in~$Y'$ and contained in a good path.
  Lemma~\ref{getapath1} gives us a good path starting at any vertex of~$Y'$, thus any such vertex is incident with at least one edge in~$S$.
  To complete the proof it suffices to prove that there is a vertex in~$Y'$ incident with two such edges.
  Accordingly, we now assume (for a contradiction) that every vertex in~$Y'$ is incident with precisely one edge in~$S$.

\medskip

\noindent{\it Claim: } There exists a path $Q \subseteq H'$ with ends $z',y'$ and interior vertex $y$ such that $y,y' \in Y'$ and the edge of $yQy'$ that is incident with $y$ is not in $S$.

\smallskip

\noindent{\it Proof of the claim:}
   Call a cycle $C \subseteq H'$ \emph{obliging} if it contains distinct vertices $y,y' \in Y'$ with the property
   that one of the two paths in $C$ with ends $y,y'$ contains the edge of $S$ incident with $y$, and the other path contains the edge in $S$ incident with $y'$.
   If $C$ is obliging, we may choose a (possibly trivial) path from $z'$ to $V(C)$ and this path together with $C$ will contain a path satisfying the claim.  Thus we may assume no cycle is
   obliging.  Note that this implies that every cycle contains at most two  vertices of $Y'$.
   Choose a cycle $C$ containing two distinct vertices, say $y_1, y_2 \in Y'$  and then choose $y_3 \in Y' \setminus V(C)$.  Since $H'$ is 2-connected, we may
   choose a  path $P_3$ internally disjoint from $C$ so that both ends of $P_3$ are in $V(C)$ and $y_3$ is an internal vertex of $P_3$.
   Let $w,w'$ be the ends of $P_3$ and for $i=1,2$ let $P_i$ be the path of $C$ with ends $w,w'$ that contains $y_i$.
   Now we must have $C \cup P_3 = P_1 \cup P_2 \cup P_3$ or cycle $P_1 \cup P_3$ contains three vertices~$y_1$, $y_2$,~$y_3$ of~$Y'$.
   Moreover, since there is no obliging cycle, by possibly interchanging $w$ and~$w'$ we may assume that $wP_iy_i$ avoids~$S$ for $i=1,2,3$.
   Finally, choose a (possibly trivial) path of $H'$ from $z'$ to $V(P_1 \cup P_2 \cup P_3)$ and observe that this path together with $P_1 \cup P_2 \cup P_3$ contains a path satisfying the claim.~$~\Box$

\medskip

  Now apply the claim to choose a path $Q$ and vertices $y,y'$.  Let $Q' = z'Qy$ and note that the unique edge of~$S$ incident with~$y$ is contained in~$Q'$.
  Now we will take advantage of~$Q'$ to construct another good path.
  Thanks to the presence of the path $Q$ we may choose a path $P \subseteq H' - E(Q')$ so that $P$ has $y$ as one end and the other end in $Y' \setminus \{y\}$ and subject to this we choose $P$ so that:
\begin{enumerate}[label=(\roman*)]
  \item The component of $H' - E(P)$ containing $Q'$ has maximum size.
  \item Subject to (i), the lexicographic ordering of the sizes of the components of $H' - E(P)$ not containing~$Q'$ is maximum (i.e., the largest component not containing~$Q'$ has maximum size,  subject to this
    the second largest has maximum size, and so on).
\end{enumerate}
We claim that the resulting path $P$ will be good.  Suppose otherwise and let $F$ be the smallest component of $H' - E(P)$ not containing $Q'$.  Note that $F$ cannot be an isolated vertex,
   since that would be an interior vertex of~$P$ in~$Y'$ -- and we could shorten $P$ to end at this vertex, improving our criteria.

   Let $P'$ be the minimal subpath of~$P$ containing all vertices of~$F\cap V(P)$.
   If there is another component of $H' - E(P)$ containing a vertex in $V(P')$ then
   we may choose a path $P^* \subseteq F$ with the same ends as $P'$ and modify $P$ by replacing the subpath $P'$ by $P^*$ to obtain a path superior to $P$ thus contradicting our choice.
   (Note that $P^*$~is disjoint from~$Q'$ by the choice of~$F$.)
   Thus, all vertices in $P'$ belong to $F$.
   There must exist a vertex $u \in Y \cap V(F)$. 
   Otherwise, the two edges incident with the ends of $P'$ that are not in $P'\cup F$ would form a 2-edge-cut in $G$.
   Now we may choose a path from $V(P')$
   to $u$ and reroute $P$ using this path.  This will result in a path superior to $P$ thus contradicting our choice.   This proves that $H' - E(P)$ is connected, so $P$ is a good path.  This gives us a
   contradiction, since now both edges of~$H'$ incident with $y$ are contained in~$S$.  This completes the proof.
\end{proof}

Our last lemma provides a technical property that we will use to control the behaviour of our process.

\begin{lemma}
\label{nottoolong}
  Let $G = (V,E)$ be a 3-edge-connected cubic graph, and let $X \subseteq V$ have $G[X]$ connected.  Let $H$ be a component of $G - X$ and let $P \subseteq H$ be a nontrivial path with ends $y_1,y_2$.  If $P$ is
  peripheral in $H$, $E(y_i,X) \neq \emptyset$ for $i=1,2$, and $E(y,X) = \emptyset$ for all other vertices $y\in V(P)$, then there exists a peripheral cycle $C$ of $G$ with $C \cap H = P$.
\end{lemma}

\begin{proof}
  For $i=1,2$ let $x_i y_i \in E(y_i,X)$.  By assumption there exists a path in $G$ from $x_1$ to $x_2$, say $Q$, so that $E(Q) \cap E(H) = \emptyset$.  Among all such paths, choose one so that:
\begin{enumerate}[label=(\roman*)]
  \item The component of $G - E(Q)$ containing $E(H)$ has maximum size.
  \item Subject to (i), the lexicographic ordering of the sizes of the components of $G - E(Q)$ not containing $E(H)$ is maximum.
\end{enumerate}
Suppose (for a contradiction) that $Q$ is not peripheral and let $F$ be a minimum size component of $G - E(Q)$ not containing $H$.
  If $Q'$ is the minimum subpath of $Q$ containing all vertices in $F$, then $V(Q') \not\subseteq V(F)$ as otherwise $G$ would have just two edges between $V(F)$ and the other vertices.
  However we may then choose a path $Q^* \subseteq F$ with the same ends as $Q'$ and then modifying $Q$ by replacing the subpath $Q'$ with $Q^*$ gives us an improvement to $Q$.
  Therefore, our chosen path~$Q$ is peripheral. Moreover, the path $P$~is peripheral in~$H$ and $G$ is 3-edge-connected.
  It follows that the cycle $C$ consisting of $P \cup Q$ together with the edges $x_1 y_1$ and $x_2 y_2$ is peripheral in $G$ and this completes the proof.
\end{proof}

\section{Peripheral 2-bases}
\label{sec:7}

Jaeger, Linial, Payan, and Tarsi \cite{JLPT} found an alternative proof of Seymour's 1-base and 2-base decomposition theorem.
Their theorem is slightly sharper than Seymour's in that it saves a vertex (a feature we will need). 
Recall that an edge-set $F \subseteq E(G)$ being peripheral in the graph $G$ means that $G-F$ is connected, in other words the set $E-F$ contains edge-set of some spanning tree of~$G$.

\begin{theorem}[Jaeger, Linial, Payan, and Tarsi \cite{JLPT}]\label{thmJLPT}
If $G$ is a graph obtained from a 3-connected cubic graph by deleting a single vertex then $G$ has a peripheral 2-base.
\end{theorem}

The proof of the above theorem in \cite{JLPT} is based on an inductive approach applied to the class of graphs which are a single-vertex deletion from a cubic 3-connected graph.  For our purpose we will adopt a blend of these
ideas.  We will operate iteratively following Seymour, but we will save a vertex like Jaeger et al.  For any graph $G = (V,E)$ and $E' \subseteq E$ we let $V(E')$ denote the set of vertices of $G$ incident to some edge in $E'$.

\begin{lemma}
\label{lem:7.2}
For every 3-edge-connected cubic graph $G$, the following holds:
\begin{enumerate}
  \item If $C\subseteq G$ is a peripheral cycle, there exists a peripheral 2-base $B\subseteq E(G)$ with $E(C)\subseteq B$.
  \item For every $r \in V(G)$, the graph $G-r$ has at least three peripheral 2-bases.
\end{enumerate}
\end{lemma}

\begin{proof}
Although the two parts to the lemma have slightly different inputs, we will prove both simultaneously.  In the first case, choose $f \in E$ to be an edge with exactly one endpoint in $V(C)$.  For the second case,
  let $f$ be an edge incident with $r$ and apply Theorem~\ref{tutteperiphcubic} to choose a peripheral cycle $C$ so that $r \in V(C)$ but $f \not\in E(C)$.  Now for both parts of the proof we will use the cycle
  $C$ and the edge $f$ and construct two sequences of nested subsets. The first one are nested edge-sets
  $B_0 \subseteq B_1 \subseteq \cdots \subseteq B_t \subseteq E$, the second one are nested vertex-sets
  $X_0 \subseteq X_1 \subseteq \cdots \subseteq X_t = V$, where $B_0 = E(C)$ and $X_0 = V(C)$.
For every $0 \le i \le t$ we will maintain the following properties:
\begin{enumerate}[label=(\roman*)]
    \item $B_i \subseteq E(X_i)$.
    \item $B_i$ is peripheral in $G$.
    \item The graphs $G[X_i]$ and $G - X_i$ are connected (or empty).
    \item $\langle B_i \rangle_2$ contains every edge with both ends in $X_i$.
\end{enumerate}
  Note that the initial sets $B_0$ and $X_0$ satisfy (i)--(iv) for $i=0$. Assuming $V(B_i) \neq V$ we form the next sets as follows: Apply Lemma~\ref{getapath1} to $G$ with the
  set $X_i$, the edge $f$, and $H$ the unique component of $G - X_i$.  If $P$ is the path selected by this lemma, we define $B_{i+1} = B_i \cup E(P)$ and $X_{i+1} = X_i \cup V(P)$.  Observe that all four of the
  above properties are still satisfied.  We continue this process until $X_t = V$ at which point the set $B_t$ is a peripheral 2-base.  This finishes the proof of the first part of the lemma.  To complete the
  proof of the second part we will investigate the behaviour of the edge $f$ in our process.  Let $f = rz$ where $r \in V(C)$ and observe that the vertex $z$ cannot appear in the path $P$ selected by
  Lemma~\ref{getapath1} (because this path is peripheral in $H$ and no intermediate vertex on the path is adjacent to $X_i$) until $H$ is just the single isolated vertex $z$.
  It follows that $B_t \setminus \delta(r)$ is a peripheral 2-base in the graph $G - r$.
  If $z,z',z''$ are the vertices adjacent to $r$ in~$G$, then $B_t \setminus \delta(r)$ will contain an edge incident with $z'$ and one incident with $z''$ but none incident with $z$.  Since $f \in \delta(r)$ may be
  chosen arbitrarily, we have found three 2-bases in $G - r$ as desired.
\end{proof}

The above lemma gives us peripheral 2-bases with a couple of useful properties.  However, in order to  prove our main theorem about $\Z_6$-flows we require the existence of many peripheral  2-bases.  This is
achieved by the following lemma.

\begin{lemma}
\label{manydecomps}
  Let $G$ be a 3-connected cubic graph on $n$ vertices with a distinguished root $r \in V(G)$.  If every peripheral cycle of $G$ has length at most $q$, then the graph $G-r$ has at least $2^{n/(2q)}$
  decompositions into a spanning tree and a 2-base.
\end{lemma}

\begin{proof}
  We proceed by induction on $n$.  As a base case, observe that when $n \le 2q$, the result follows immediately from the previous lemma.  For the inductive step we begin by considering the case that there exists a
  partition $\{X_1, X_2\}$ of $V$ with $|X_i| \ge 2$ for $i=1,2$ and $|E(X_1,X_2)| = 3$.  We may assume that $r \in X_1$ and for $i=1,2$ form a graph $G_i$ from $G$ by identifying $X_i$ to a single vertex
  called $x_i$ and deleting any loops formed in this process. It is easy to see that $G_1$ and $G_2$ are 3-connected cubic graphs. Let $e,e' \in E(X_1,X_2)$ be distinct and for $i=1,2$ let $C_i$ be a peripheral cycle of $G_i$ that contains $e,e'$.  Now the cycle of $G$ formed from
  the union of $C_1 - x_1$ and $C_2 - x_2$ by adding the edges $e,e'$ is a peripheral cycle of $G$.  It follows from this and Theorem~\ref{tutteperiphcubic} that neither $G_1$ nor $G_2$ has a peripheral cycle with
  length greater than $q$.  So, by the induction hypothesis, there are at least $2^{|V(G_1)|/(2q)}$ peripheral 2-bases of $G_1 - x_1$ and at least $2^{|V(G_2)|/(2q)}$ peripheral 2-bases of $G_2 - r$.
  The union of a peripheral 2-base of $G_1 - x_1$ with a peripheral 2-base of $G_2 - r$ is a peripheral 2-base of $G-r$ and this gives the desired count.

So we may now assume that $G$ is cyclically 4-edge-connected.  Now we will show that we have many degrees of freedom in selecting a peripheral 2-base using a procedure similar to that used in the proof of Lemma~\ref{manydecomps}.
We construct two sequences of nested subsets, edge-sets
$B_0 \subseteq B_1 \subseteq \cdots \subseteq B_t \subseteq E$ and
vertex-sets $X_0 \subseteq X_1 \subseteq \cdots \subseteq X_t = V$.
For every $0 \le i \le t$ we will maintain the same properties (i)--(iv) as in the proof of the previous lemma.

  We begin by choosing a peripheral cycle $C$ containing $r$ (note that we have three ways to do this).  Let $B_0 = E(C)$, let $X_0 = V(C)$ and we let  $\{f\} = \delta(r)  \setminus B_0$.
  Now at each step assuming $X_i \neq V$ we operate as follows:  If there is a vertex $y \in V \setminus X_i$ so that $|E(y,  X_i ) \setminus  \{f\}| \ge  2$
  then we let $B_{i+1} = B_i$ and let  $X_{i+1} = X_i \cup \{y\}$ (we have added a trivial path of one new vertex and no new edge to the 2-base).
  If no such vertex exists, then we apply Lemma~\ref{getapath2} to choose a vertex $y \in V \setminus  X_i$ and peripheral paths $P_1, P_2$.
  Now we can choose to either set $X_{i+1} = X_i \cup V(P_1)$ and $B_{i+1} = B_i \cup E(P_1)$ or we may set $X_{i+1} = X_i \cup V(P_2)$ and $B_{i+1} = B_i \cup E(P_2)$.
  We continue the process until we have $X_t = V$.

  In order to see that this operation gives us the desired flexibility, it is helpful to introduce another nested sequence of edges $T_0 \subseteq T_1 \subseteq \cdots \subseteq T_t$ defined by the rule
  $T_i = E(X_i) \setminus B_i$.  The key feature of these sets (verified by a straightforward induction) is that for every $1 \le i  \le t-1$, the set $T_i  \cup E(X_i, V \setminus X_i)$ is a  spanning tree in the
  graph obtained from $G$ by identifying $V \setminus X_i$ to a single vertex.  For every $1 \le i \le t-1$ we have $|T_{i} \setminus T_{i-1}| = 2$.  At the last step we have
  $|T_t \setminus T_{t-1}| = 3$ and the set  $T_t$ forms the edge-set of a spanning tree in $T$.  Therefore, $|V| - 1 = |T_t| =  2(t-1) + 3$ and we have $t = \frac{1}{2}|V| -1$.
  It follows from this that $|B_t| = \frac{1}{2}|V| + 1$.  It follows from Lemma~\ref{nottoolong} that every path $P$ we select using Lemma~\ref{getapath2} has length at most~$q$.
  So the total number of nontrivial paths selected in our process must be at least $\frac{|V|}{2q}$.

  It remains to show that different choices of paths during our process yield different peripheral 2-bases.
  From our construction follows immediately that  $E(X_i,  V \setminus X_i) \subseteq T_t$ for every $0 \le i \le t-1$.  Suppose that when we have
  $B_i$ and $X_i$ and apply Lemma~\ref{getapath2} we select the vertex $y \in V \setminus X_i$ and  the paths $P_1$, $P_2$ (both ending at $y$).
  Let $\{e_0\} = E(y,X_i)$ and for $j=1,2$ let $e_j$ be the edge of $P_j$ incident with $y$ (the lemma gives $e_1 \neq e_2$).
  If we choose the path $P_j$ and set $B_{i+1} = B_i \cup E(P_j)$ and $X_{i+1} = X_i \cup V(P_j)$ then upon completion of our process we will
  have $\delta(y) \cap B_t = \{e_j\}$.  So the 2-bases constructed by making a different choice of $P_1$ or $P_2$ will always be distinct.
  This gives us at least $2^{n/(2q)}$ peripheral 2-bases of $G$, as desired.
\end{proof}

\section{Flows in $\Z_6$}
\label{sec:8}

In this section we will first prove a lemma that provides the existence of many peripheral 2-bases in a 3-edge-connected cubic graph with a long peripheral cycle.  We will then use this to prove our main
theorem showing the existence of many $\Z_6$-flows in the setting of group connectivity for 3-edge-connected graphs.

\begin{lemma}
  Let $G = (V,E)$ be an oriented 3-edge-connected cubic graph with a peripheral cycle $C$ with $|V(C)| = q$.  For every $f : E \rightarrow \Z_2 \times \Z_3$ there exist at least $2^{2q/3}$ flows
  $\phi: E \rightarrow \Z_2 \times \Z_3$ with $\phi(e) \neq f(e)$ for every $e \in E$.
\end{lemma}

\begin{proof}
  Put $f = (f_1, f_2)$.  Now choose a partition of $E$ into $\{B,T\}$ so that $B$ is a 2-base with $E(C) \subseteq B$ and $T$ is the edge-set of a spanning tree.
  Apply Lemma~\ref{base2groupflow} to choose a flow $\phi_2 : E \rightarrow \Z_3$ satisfying $\phi_2(e) \neq f_2(e)$ for every $e \in T$.
  By possibly modifying $\phi_2$ by adding an elementary flow around $C$, we may further assume that $A = \{ e \in E(C) \mid \phi_2(e) \neq  f_2(e)\}$ satisfies $|A| \ge \frac{2}{3}|E(C)| \ge \frac{2}{3}q$.
  Let $B' = \{ e \in B \setminus A \mid f_1(e) = 0 \}$, and for every $ e\in B$, let $C_e$ be the edge-set of the fundamental cycle of $e$ with respect to the spanning tree $(V,T)$.
  Now for every set $S$ with $B' \subseteq S \subseteq B' \cup A$, there is a $\Z_2$-flow $\phi_1$ with support $\bigoplus_{e \in S} C_e$ and the $\Z_2 \times \Z_3$-flow $\phi = (\phi_1, \phi_2)$ satisfies $\phi(e) \neq f(e)$ for every $e \in E$.
  There are $2^{|A|}$ choices for $S$, each of which gives a different flow.
  Since $|A| \ge \frac{2}{3}q$, this gives the desired bound.
\end{proof}

With this last lemma in place we are ready to prove Theorem~\ref{order6}, our main theorem concerning flows in $\Z_6$. We have restated it for convenience.

\begin{theorem*}
  Let $G=(V,E)$ be an oriented 3-edge-connected graph with $\ell = |E|-|V| \ge 11$, and let $f : E \rightarrow \Z_2 \times \Z_3$.
  There exist at least $2^{ \sqrt{\ell} / \log \ell}$ flows $\phi : E \rightarrow \Z_2 \times \Z_3$ with the property that $\phi(e) \neq f(e)$ for every $e \in E$.
\end{theorem*}

\begin{proof}
  Apply Lemma~\ref{reduce2cubic} to choose a cubic graph $G' = (V',E')$ with $|V'| = 2 \ell$ so that $G$ can be obtained from $G'$ by contracting the edges of a forest.  Extend the function $f$ to have domain $E'$
  arbitrarily, and let $f = (f_1, f_2)$ where $f_1 : E' \rightarrow \Z_2$ and $f_2 : E' \rightarrow \Z_3$.  Observe that by Theorem~\ref{forestdiff} it suffices to prove the result for $G'$ in place of $G$.

  If $G'$ has a peripheral cycle of length at least $q=\frac{3}{2} \sqrt{\ell} / \log \ell$, then the result follows immediately from the previous lemma.  Otherwise, by Lemma~\ref{manydecomps}, we can choose
  $N = \bigl\lceil 2^{2\ell/(2q)} \bigr\rceil = \bigl\lceil 2^{ \frac{2}{3} \sqrt{\ell} \log \ell} \bigr\rceil$ pairwise distinct partitions of~$E'$, say $\{T_1, B_1\}, \ldots, \{T_{N}, B_{N}\}$ where for $1 \le i \le N$ the set $T_i$ is the edge-set of a spanning
  tree and $B_i$ is a 2-base. (Note that $|B_i| = \ell+1$.)
  For each such partition $\{T_i,B_i\}$, $1 \le i \le N$, we apply Lemma~\ref{base2groupflow} to choose a $\Z_3$-flow $\phi_i : E' \rightarrow \Z_3$ satisfying
  $\phi_i(e) \neq f_2(e)$ for every $ e \in E' \setminus B_i$.

  First suppose there exists $1 \le i \le N$ for which the flow $\phi_i$ has the property that $\phi_i(e) \neq f_2(e)$ holds for at least $\sqrt{\ell} / \log \ell$ edges $e \in B_i$.  In this case, we may proceed
  as in the proof of the previous lemma to construct $2^{\sqrt{\ell} / \log \ell }$ flows:  Set $A = \{e \in  B_i \mid \phi_i(e) \neq f_2(e) \}$, set $B' = \{  e \in B_i \setminus A \mid f_1(e) = 0 \}$, and then
  for every $B' \subseteq S \subseteq B' \cup A$ form a $\Z_2$-flow with support $\bigoplus_{e \in S} C_e$ (where $C_e$ is the edge-set of the fundamental cycle of $e$ with respect to the tree $T_i$) and combine
  this with $\phi_i$ to get at least $2^{\sqrt{\ell} / \log \ell }$ valid flows.

  Thus we may assume that every $\phi_i$ has at most $\sqrt{\ell} / \log \ell$ edges $e \in B_i$ for which $\phi_i(e) \neq f_2(e)$.  This means that each flow $\phi_i$ will agree with the function $f_2$ on all but
  at most $\sqrt{\ell} / \log \ell $ elements of $B_i$ and on no elements in $T_i$.  In particular, we have
\begin{equation}
\label{eq:bound}
  \ell + 1 - \sqrt{\ell} / \log \ell  \le | \{ e \in E' \mid \phi_i(e) = f_2(e) \}| \le \ell+1.
\end{equation}
Now let $\nu : E' \rightarrow \Z_3$ be a flow, let $A = \{e \in E' \mid \nu(e) = f_2(e) \}$ and let $r := \ell + 1 - |A|$.  We will find an upper bound on the number of indices $1 \le i \le N$
for which $\phi_i = \nu$.  If $r < 0$ or $r > \sqrt{\ell} / \log \ell$, then (\ref{eq:bound}) shows that $\nu \neq \phi_i$ for every $1 \le i \le N$.  Otherwise, in order for $\nu = \phi_i$ it must be that the 2-base $B_i$ consists of all
of the edges in $A$ plus $r$ edges from $E' \setminus A$.  The number of ways to select such a set is equal to $\binom{|E\setminus A|}{r}$, which we further estimate using the bound ${ n \choose k } \le \left(\frac{en}{k} \right)^k$ and the fact that the function $k \mapsto ( \frac{en}{k} )^k$ is increasing for $k < n$:
$$
   \binom{|E\setminus A|}{r} = { 2\ell-1 + r \choose r}
       \le \left( \frac{ 3e\ell }{r} \right)^{r}
       \le \left( \frac{ 3e\ell }{ \sqrt{\ell} / \log \ell } \right)^{ \sqrt{\ell} / \log \ell}
        =  \left(3 e \sqrt{ \ell } \log \ell \right)^{ \sqrt{\ell}/\log \ell }.
$$
It follows that the number of distinct flows in our list $\phi_1, \ldots, \phi_N$ is at least
\[
   \frac{2^{ \frac{2}{3} \sqrt{\ell} \log \ell } }{ 2^{ \left( \sqrt{ \ell} / \log \ell \right) \log\left(3e \sqrt{\ell} \log \ell \right) } }
      	= 2^{( \sqrt{\ell} / \log \ell) \left( \frac{2}{3}(\log \ell)^2 - \log(3 e \sqrt{ \ell} \log \ell) \right)} .
\]
Since $\ell \ge 11$, we have $\frac{2}{3}(\log \ell)^2 -\log(3 e \sqrt{\ell} \log \ell) \ge 1$, and our list $\phi_1, \phi_2, \ldots, \phi_N$ contains at least $2^{ \sqrt{\ell} / \log \ell}$ distinct flows.
For every $1 \le i \le N$ we may apply Lemma~\ref{base2groupflow} to choose a flow $\psi_i : E' \rightarrow \Z_2$ so that $\psi_i(e) \neq f_1(e)$ holds for every $e \in B_i$.
So, every $(\psi_i, \phi_i)$ is a $\Z_2 \times \Z_3$ flow for which $(\psi_i(e), \phi_i(e)) \neq (f_1(e), f_2(e))$ holds for  every $e \in E'$ and we have at least $2^{ \sqrt{\ell} / \log \ell}$ such flows, thus
completing our proof.
\end{proof}

\section{Open problems} 

As our main question, we would like to know what is that status of Conjecture~\ref{conjexp} for $\Z_6$ and $\Z_7$. 
However, we also want to list here some further questions that came up during our work on this paper. 

We conjecture that a 3-edge-connected, nonplanar graph with representativity at least 5 has exponentially many peripheral cycles.
Note that for 3-edge-connected planar graphs, peripheral cycles are exactly facial walks, so there is at most $2n-4$ of them.

As mentioned before, a result of Jaeger et al.~\cite{JLPT} (Theorem~\ref{thmJLPT}) gives a decomposition 
of a graph obtained from a 3-connected cubic graph by deleting a single vertex into a 1-base and a 2-base. 
We conjecture that such graph can also be decomposed into three 2-bases. 

\bibliographystyle{acm}
\bibliography{many_flows}

\end{document}